\documentclass[11pt, a4paper]{article}

\usepackage[english]{babel}
\usepackage{eucal}
\usepackage[T1]{fontenc}

\usepackage{amsmath}
\usepackage{amssymb}
\usepackage{amsthm}

\usepackage{tikz}
\usepackage{tikz-cd}

\usepackage{hyperref}
\usepackage[capitalize]{cleveref}
\usepackage{footmisc}
\usepackage{bbm}
\usepackage{mathrsfs}
\usepackage{extarrows}
\usepackage{enumitem}
\usepackage{stmaryrd}
\usepackage{microtype}

\newcommand{\fin}{\mathrm{fin}}

\newcommand{\gl}{\mathrm{gl}}

\DeclareMathOperator{\res}{res}

\DeclareMathOperator{\TMF}{TMF}

\DeclareMathOperator{\CAlg}{CAlg}

\DeclareMathOperator{\Cat}{Cat}

\DeclareMathOperator{\Fun}{Fun}

\newcommand{\Glo}{\mathcal{G}}

\DeclareMathOperator{\Span}{Span}

\DeclareMathOperator{\Sp}{Sp}

\newcommand{\spaces}{\mathcal{S}}

\DeclareMathOperator{\Out}{Out}

\newcommand{\Fin}{\mathbb{F}}

\DeclareMathOperator{\Aut}{Aut}

\DeclareMathOperator{\laxlim}{laxlim}
\newcommand{\id}{\mathrm{id}}

\DeclareMathOperator{\Hom}{Hom}

\newcommand{\op}{\mathrm{op}}

\newcommand{\bbG}{\mathbb{G}}
\newcommand{\bbL}{\mathbb{L}}

\newcommand{\E}{\mathbf{E}}

\newcommand{\bbE}{\mathbb{E}}

\newcommand{\calF}{\mathcal{F}}

\newcommand{\calG}{\mathcal{G}}

\newcommand{\h}{\mathrm{h}}

\newcommand{\calN}{\mathcal{N}}

\newcommand{\bbO}{\mathbb{O}}

\newcommand{\Q}{\mathbf{Q}}

\newcommand{\bbN}{\mathbb{N}}

\newcommand{\Z}{\mathbf{Z}}

\newcommand{\calC}{\mathcal{C}}

\newcommand{\calE}{\mathcal{E}}

\newcommand{\bbZ}{\mathbb{Z}}

\newcommand{\bbF}{\mathbb{F}}

\newcommand{\bbC}{\mathbb{C}}
\newcommand{\bbR}{\mathbb{R}}
\newcommand{\calO}{\mathcal{O}}
\newcommand{\calW}{\mathcal{W}}

\newcommand{\bbI}{\mathbb{I}}

\newcommand{\family}{\mathcal{F}}

\newcommand{\norm}{\mathrm{norm}}

\newcommand{\calI}{\mathcal{I}}

\newcommand{\dcat}{\mathcal{D}}
\newcommand{\ccat}{\mathcal{C}}

\newcommand{\spherespectrum}{\mathbf{S}}

\newcommand{\bs}{{-}}
\newcommand{\ul}[1]{\underline{#1}}

\newcommand{\Grpd}{\bbG}
\newcommand{\UCom}{\mathrm{UCom}}
\newcommand{\Comm}{\mathrm{Comm}}

\newcommand{\inert}{\mathrm{int}}

\DeclareMathOperator{\Ar}{Ar}

\DeclareMathOperator{\Corr}{Corr}

\DeclareMathOperator{\Orb}{\mathcal{O}}

\usepackage{mathtools}

\newtheorem*{theorem*}{Theorem}

\theoremstyle{plain}\numberwithin{equation}{section}
\newtheorem{theorem}[equation]{Theorem}
\Crefname{theorem}{{Th}.\!\!}{{Ths}.\!\!}
\newtheorem{theoremalph}{Theorem}

\Crefname{theoremalph}{{Th}.\!\!}{{Ths}.\!\!}
\Crefname{coralph}{{Cor}.\!\!}{{Cors}.\!\!}

\Crefname{defalph}{{Df}.\!\!}{{Dfs}.\!\!}

\Crefname{conjalph}{{Conj}.\!\!}{{Conjs}.\!\!}

\Crefname{problem}{{Prb}.\!\!}{{Prbs}.\!\!}
\newtheorem{prop}[equation]{Proposition}
\Crefname{prop}{{Pr}.\!\!}{{Prs}.\!\!}
\newtheorem{lemma}[equation]{Lemma}
\Crefname{lemma}{{Lm}.\!\!}{{Lms}.\!\!}

\Crefname{cor}{{Cor}.\!\!}{{Cors}.\!\!}

\Crefname{conjecture}{{Conj}.\!\!}{{Conjs}.\!\!}

\theoremstyle{definition}\numberwithin{equation}{section}
\newtheorem{mydef}[equation]{Definition}
\Crefname{mydef}{{Df}.\!\!}{{Dfs}.\!\!}

\Crefname{defn}{{Df}.\!\!}{{Dfs}.\!\!}

\Crefname{recall}{{Rcl}.\!\!}{{Rcls}.\!\!}

\Crefname{construction}{{Con}.\!\!}{{Cons}.\!\!}

\Crefname{ass}{{As}.\!\!}{{As}.\!\!}

\Crefname{notation}{{Nt}.\!\!}{{Nts}.\!\!}

\Crefname{situation}{{St}.\!\!}{{Sts}.\!\!}

\theoremstyle{remark}\numberwithin{equation}{section}
\newtheorem{example}[equation]{Example}
\Crefname{example}{{Ex}.\!\!}{{Exs}.\!\!}
\newtheorem{ex}[equation]{Example}
\Crefname{ex}{{Ex}.\!\!}{{Exs}.\!\!}

\Crefname{nonexample}{{NonEx}.\!\!}{{NonEx}.\!\!}

\Crefname{claim}{{Clm}.\!\!}{{Clms}.\!\!}
\newtheorem{remark}[equation]{Remark}
\Crefname{remark}{{Rmk}.\!\!}{{Rmks}.\!\!}

\Crefname{rmk}{{Rmk}.\!\!}{{Rmks}.\!\!}

\Crefname{idea}{{Id}.\!\!}{{Ids}.\!\!}

\Crefname{warn}{{Warn}.\!\!}{{Warns}.\!\!}

\Crefname{question}{{Qn}.\!\!}{{Qns}.\!\!}

\Crefname{figure}{{Fig.}\!\!}{{Figs.}\!\!}
\Crefname{footnote}{{Fn.}\!\!}{{Fn.}\!\!}

\Crefname{part}{{\textsection}\!\!}{{\textsection}\!\!}
\Crefname{chapter}{{\textsection}\!\!}{{\textsection}\!\!}
\Crefname{section}{{\textsection}\!\!}{{\textsection}\!\!}
\Crefname{subsection}{{\textsection}\!\!}{{\textsection}\!\!}
\Crefname{appendix}{{\textsection}\!\!}{{\textsection}\!\!}

\usepackage{parskip}

\makeatletter
\pretocmd{\remark}{\vspace{-\parskip}}{}{}
\pretocmd{\example}{\vspace{-\parskip}}{}{}
\makeatother

\begingroup
    \makeatletter
    \@for\theoremstyle:=mydef,remark,plain,theorem,defn\do{%
        \expandafter\g@addto@macro\csname th@\theoremstyle\endcsname{%
            \addtolength\thm@preskip\parskip
            }%
        }
\endgroup

\usetikzlibrary{spath3}
\tikzset{between/.style n args={2}{/tikz/spath/at end path construction={
    \tikzset{spath/split at keep middle={current}{#1}{#2}}
}}}

\begin{document}
\title{An algebraic model for rational ultracommutative rings}

\author{William Balderrama\footnote{\url{williamb@math.uni-bonn.de}}, Jack Morgan Davies\footnote{\url{davies@uni-wuppertal.de}}, and Sil Linskens\footnote{\url{sil.linskens@mathematiks.uni-regensburg.de}}}
\maketitle

\begin{abstract}
Given a global equivariant ultracommutative ring spectrum $E$ and inclusion $H\hookrightarrow G$ of finite groups, one may apply geometric fixed points to the norm $N_H^G E_H \to E_G$ to obtain what we call a \emph{geometric norm} $\Phi^H E \to \Phi^G E$. We prove that, together with inflations, these assemble into a functor $\Phi\colon \mathrm{UCom}_{\mathrm{fin}} \to \Fun(\Span(\mathcal{G},\mathcal{E},\mathcal{O}),\mathrm{CAlg})$, where $\Span(\mathcal{G},\mathcal{E},\mathcal{O})$ is the span category of finite connected groupoids with full backwards maps and faithful forwards maps, and that $\Phi$ restricts to an equivalence between full subcategories of rational objects.

Central to our construction is a refinement of geometric fixed points to a natural transformation $\Phi\colon \Sp_\bullet\to\Fun(\Orb_\bullet^\simeq,\Sp)$ which is compatible with restrictions and norms, and which restricts to an equivalence on full subcategories of rational objects. We explain how this may also be used to recover theorems of Barrero--Barthel--Pol--Strickland--Williamson and Wimmer on algebraic models for rational global spectra and normed $G$-commutative ring spectra respectively. 
\end{abstract}

\setcounter{tocdepth}{2}
\tableofcontents

%%%%%%%%%%%%%%%%%%%%%%%%%%%%%%%%%%%%%%%%%%%%%%%%%%%%%%%%
%%%%%%%%%%%%%%%%%%%%%%%%%%%%%%%%%%%%%%%%%%%%%%%%%%%%%%%%
%%%%%%%%%%%%%%%%%%%%%%%%%%%%%%%%%%%%%%%%%%%%%%%%%%%%%%%%
\section{Introduction}

\emph{Global equivariant homotopy theory} is the study of equivariant phenomena in homotopy theory that occurs simultaneously for all groups in a compatible way. It has its origins in the observation that many fundamental examples of equivariant cohomology theories are global: equivariant $K$-theory, equivariant bordism, equivariant stable homotopy, Borel cohomology, and more. These are now well understood to be examples of \emph{global spectra}. We refer the reader to \cite{s} for a textbook account of global homotopy theory; we restrict ourselves to global homotopy theory for finite groups here.

Associated to a global spectrum $E$ is an underlying genuine $G$-spectrum $E_G$ for every finite group $G$, which as $G$ varies are related by compatible equivalences $\res^G_H E_G\simeq E_H$ for each inclusion $H \hookrightarrow G$ and inflations $\operatorname{infl}_K^G E_K \to E_G$ for each surjection $G \twoheadrightarrow K$. This structure endows the geometric fixed points $\Phi^GE = \Phi^G(E_G)$ of a global spectrum with inflations:
\[
G\twoheadrightarrow K\quad\leadsto\quad \Phi^K E \simeq \Phi^G(\operatorname{infl}_K^GE_K) \to \Phi^G E.
\]
Additional structure on an equivariant spectrum $E$ is reflected in additional structure on its geometric fixed points. For example, if $E$ is a global ring spectrum, then its geometric fixed points are ring spectra, and the above inflations $\Phi^K E \to \Phi^G E$ are multiplicative.

The fundamental examples of highly structured commutative rings in global homotopy theory are the \emph{ultracommutative rings} \cite[\textsection 5]{s}. Underlying an ultracommutative ring $E$ is a $G$-commutative ring $E_G$ in the sense of Hill--Hopkins--Ravenel \cite{hhr} for every finite group $G$. In particular, associated to each inclusion $H\hookrightarrow G$ is an equivariant norm map $N_H^G E_H \to E_G$. On geometric fixed points, these induce what we call \emph{geometric norms}: multiplicative maps
\[
H \hookrightarrow G \quad \leadsto\quad\Phi^H E \simeq \Phi^G(N_H^G E_H) \to \Phi^G E
\]
of $\E_\infty$ ring spectra. 
The goal of this article is to provide means of capturing this kind of additional structure on geometric fixed points, and to show that, rationally, it is everything.

In the case of ultracommutative rings, our main theorem is as follows.
Let $\calG\subset\spaces$ be the full subcategory of spaces spanned by the connected groupoids with finite isotropy groups (i.e.\ classifying spaces of finite groups), $\calE\subset\calG$ be the wide subcategory of full functors (i.e.\ $\pi_1$-surjections), and $\calO\subset\calG$ be the wide subcategory of faithful functors (i.e.\ $\pi_1$-injections).

\begin{theoremalph}\label{intro:main}
Geometric fixed points of ultracommutative rings assemble into a functor
\[
\Phi\colon \UCom_\fin \to \Fun(\Span(\calG,\calE,\calO),\CAlg),\qquad (\Phi E)(BG) = \Phi^G E,
\]
with functoriality of $\Phi E$ in $\calE^{\op}$ given by inflation and in $\calO$ given by geometric norms. This assignment restricts to an equivalence
\[
\UCom_{\fin,\Q} \simeq \Fun(\Span(\Glo,\calE,\calO),\CAlg_\Q)
\]
between full subcategories of rational objects.
\end{theoremalph}

This resolves a conjecture of Wimmer in his thesis \cite[p.7]{wimmerthesis}, extending his description of rational global power functors. More generally, our proof establishes the corresponding theorem for $\calN$-normed $\family$-global rings, where $\family$ is a global family of finite groups and only participating primes need be inverted, and $\calN$ is a choice of group homomorphisms that one allows norms along; see \Cref{thm:rationalucomm}.

The category $\Fun(\Span(\Glo,\calE,\calO),\CAlg_\Q)$ is algebraic: as the mapping spaces in $\Span(\Glo,\calE,\calO)$ are rationally discrete, every functor $\Span(\Glo,\calE,\calO) \to \CAlg_\Q$ factors uniquely through the homotopy category $\Out^\op_{\norm} \coloneqq \h\Span(\Glo,\calE,\calO)$, and $\Fun(\Span(\Glo,\calE,\calO),\CAlg_\Q)$ is equivalent to the underlying $\infty$-category of the model category of $\Out^\op_\norm$-indexed diagrams in rational commutative DGAs. Our work thus fits into the larger program, championed by Greenlees and coauthors, of finding algebraic models for categories of rational equivariant spectra. We refer the reader to \cite{intro_algebraic_models} for an introduction to this subject. Our own interest in the geometric fixed points of ultracommutative rings arose from ongoing work on \emph{equivariant elliptic cohomology}, in particular the conjecture that equivariant $\TMF$ admits the structure of a global ultracommutative ring. In forthcoming work \cite{geometricnorms}, we construct geometric norms on equivariant $\TMF$ out of isogenies of oriented elliptic curves. When combined with \cref{intro:main}, this will show that rational equivariant $\TMF$ admits the structure of a global ultracommutative ring for finite groups.

Let us indicate the basic idea behind our proof of \cref{intro:main}. At its core is a new analysis of the interaction between geometric fixed points and norms. If $E$ is a $G$-spectrum and $H\subset G$, then the geometric fixed points $\Phi^H E$ carry a residual action of the Weyl group $W_HG$. As $H$ varies, this assembles into a functor $\Phi\colon \Sp_G \to \Fun(\calO_{BG}^\simeq,\Sp)$, where $\calO_{BG}$ is the orbit category of $G$. Classical localization theory shows $\Phi$ is an equivalence away from the order of $G$. In \cref{sec:geometric_fixed_points}, we assemble these functors into a natural transformation
\begin{equation*}\label{eq:phi}
\Phi\colon \Sp_\bullet \Rightarrow \Fun(\calO_\bullet^\simeq,\Sp)\colon \Span(\bbG) \to \Cat.
\end{equation*}
Here, $\Span(\bbG)$ is the span category of finite groupoids, and the diagrams $\Sp_\bullet$ and $\Fun(\calO_\bullet^\simeq,\Sp)$ on $\Span(\bbG)$ encode restriction and norm functors that may be defined for these categories. As we recall in \cref{sec:review}, work of the third-named author with Lenz and P\"utzst\"uck \cite{LLPNorms} has shown that the category of global ultracommutative rings may be constructed naturally from the diagram $\Sp_\bullet$ using the theory of \emph{partially lax limits}. Carrying out the same construction with $\Fun(\calO_\bullet^\simeq,\Sp)$ in place of $\Sp_\bullet$ provides a natural target for the geometric fixed points of an ultracommutative ring. We identify this target with the functor category $\Fun(\Span(\calG,\calE,\calO),\CAlg)$ in \cref{sec:funwithlaxlims}. This establishes the first half of \cref{intro:main}, and the rational equivalence quickly follows in \cref{sec:rational}.

These techniques apply more generally to capture additional structure present on geometric fixed points for other categories of equivariant spectra and equivariant ring spectra. We illustrate this in \cref{sec:variations} with two examples. The first is a linear version of \cref{intro:main}: a geometric fixed points functor
\[
\Phi\colon \Sp^\gl_\fin \to \Fun(\calE^\op,\Sp)
\]
which restricts to an equivalence between full subcategories of rational objects. We also treat the variations for $\family$-global spectra for other global families $\family$ of finite groups, where only participating primes need be inverted. This recovers and refines previous work of Barrero--Barthel--Pol--Strickland--Williamson \cite[Th.D]{Glo_reps}, Schwede \cite[Th.4.5.13]{s}, and Wimmer \cite[Th.3.2.20]{wimmerthesis}. 
Our second example is a $G$-equivariant analogue of \cref{intro:main}: a geometric fixed points functor
\[
\Phi\colon \Comm^G \to \Fun(\calO_{BG},\CAlg)
\]
which is an equivalence away from the order of $G$. Here, $\Comm^G$ is the category of genuine $G$-commutative, or normed, ring spectra, and this recovers a theorem of Wimmer \cite[Th.1.2]{wimmerrationalmodel}. Moreover, we treat the generalization to partially normed $G$-ring spectra which, in the course of preparing this paper, has also appeared independently in a preprint of Tigilauri \cite{tigilauri2026modelnormedalgebrasrational}.

%%%%%%%%%%%%%%%%%%%%%%%%%%%%%%%%%%%%%%%%%%%%%%%%%%%%%%%%
\subsection*{Notation and conventions}
All groups of equivariance are assumed to be finite. We refer to $\infty$-categories as categories, and all constructions are to be interpreted $\infty$-categorically. A subcategory $\ccat\subset\dcat$ is a monomorphism in $\Cat$.  We employ the following notation throughout the paper:

\begin{itemize}
\item $\bbG\subset\spaces$ is the category of finite groupoids and $\Glo\subset\bbG$ is the full subcategory of connected finite groupoids, i.e.\ of classifying spaces of finite groups.
\item $\bbO\subset\bbG$ and $\calO\subset\Glo$ are the wide subcategories of faithful functors. The orbit category of $X\in\bbG$ is $\calO_X = \calO\times_{\bbG}\bbO_{/X}$. In particular $\calO_{BG}$ is equivalent, by $(BH \hookrightarrow BG)  \leftrightsquigarrow  G/H$, to the usual orbit category of $G$, and in general $\calO_{X\amalg Y}\simeq \calO_X\amalg \calO_Y$ for $X,Y\in\bbG$.
\item $\calE\subset\Glo$ is the wide subcategory of full functors and $\bbE\subset\bbG$ is its finite coproduct cocompletion. A map $f\colon Y \to X$ of finite groupoids lies in $\bbE$ when it  induces a surjection $\pi_1(Y,y) \to \pi_1(X,f(y))$ for all $y\in Y$.
\item $\family$ is a global family: a class of finite groups closed under equivalence and subgroups.
\item $\bbG_\family\subset\bbG$ is the category of groupoids with isotropy in $\family$, and $\ccat_\family = \ccat\cap\bbG_\family$ for $\ccat\subset\bbG$.
\item $\calN\subset\Glo_\family$ is an $\family$-global choice of norms (see \cref{def:choiceofnorms}) and $\bbN\subset\bbG_\family$ is its finite coproduct completion.
\item $\bbF$ is the category of finite sets, and $\bbF_X = \Fun(X,\bbF)$ for $X \in \bbG$. We note that this is equivalent, by unstraightening, to $\bbO_{/X}$.
\item $\Span(\ccat,L,R)$ is the category of spans in a category $\ccat$ with backwards and forwards arrows in given wide subcategories $L,R\subset\ccat$ closed under pullback along each other; see e.g.\ \cite{HHLNa}. We abbreviate $\Span(\ccat) = \Span(\ccat,\ccat,\ccat)$, when defined.
\item $\CAlg(\ccat)$ is the category of $\E_\infty$ algebras in a given symmetric monoidal category $\ccat$. We abbreviate $\CAlg = \CAlg(\Sp)$.
\end{itemize}

As in the above, we will often make use of the finite coproduct completion $\bbC$ of a category $\ccat$, which is seen to be equivalent to the cartesian unstraightening of the functor $\ccat^{\times\bullet}\colon \bbF^\op \to \Cat$. We call the wide subcategory $\bbC^\nabla\subset\bbC$ of cartesian arrows the \emph{fold maps}. For example, $\bbG$ and $\bbF_{X}$ are the finite coproduct completions of $\Glo$ and $\calO_X$ respectively.

Finally, we will work with symmetric monoidal categories and more general operads as defined over $\Span(\bbF)$, rather than over $\bbF_\ast$ as is done in \cite{ha}; see \cite{envelopes} for the details of this translation.

%%%%%%%%%%%%%%%%%%%%%%%%%%%%%%%%%%%%%%%%%%%%%%%%%%%%%%%%
\subsection*{Acknowledgements}
JMD was supported by the DFG-funded research training group GRK 2240: Algebro-Geometric Methods in Algebra, Arithmetic and Topology. SL was an associate member of the SFB: Higher invariants. This work was supported by EPSRC grant no EP/K032208/1.

%%%%%%%%%%%%%%%%%%%%%%%%%%%%%%%%%%%%%%%%%%%%%%%%%%%%%%%%
%%%%%%%%%%%%%%%%%%%%%%%%%%%%%%%%%%%%%%%%%%%%%%%%%%%%%%%%
%%%%%%%%%%%%%%%%%%%%%%%%%%%%%%%%%%%%%%%%%%%%%%%%%%%%%%%%
\section{Review of equivariant and global homotopy theory}\label{sec:review}
We begin with a definition for the category of equivariant spectra.

\begin{mydef}
For a finite groupoid $X\in \bbG$, the category of \emph{$X$-spectra} is the category
\[
\Sp_X \coloneqq \Fun^{\times}(\Span(\mathbb{F}_X),\Sp)
\]
of finite product-preserving functors $\Span(\bbF_X)\to\Sp$.
\end{mydef}

\begin{remark}
By the Guillou--May theorem \cite{guilloumay2024models}, $\Sp_{BG}$ is equivalent to the category $\Sp^G$ of genuine $G$-spectra.
\end{remark}

The following results encode the functoriality of $\Sp_X$ employed in this article.

\begin{mydef}
Write $\Corr\subset \Cat$ for the subcategory spanned on objects by the categories $\bbF_X$ and on morphisms by the right adjoints.
\end{mydef}

\begin{lemma}\label{lem:corr}
There is a product-preserving functor
\[
\Span(\Grpd) \to \Corr, \qquad X\mapsto \bbF_X,
\]
sending a span $X\xleftarrow{f}M\xrightarrow{g}Y$ of finite groupoids to the composite $g_\ast f^\ast$ of restriction along $f$ and right Kan extension along $g$.
\end{lemma}
\begin{proof}
This is an application of Barwick's unfurling construction; see \cite[Ex.3.4]{HHLNa}.
\end{proof}

\begin{theorem}[Yuan]\label{ex:sp_bullet}
There exists a product-preserving functor 
\[
\Sp_\bullet\colon \Corr\to \Cat,\quad X\mapsto \Sp_X,
\]
whose restriction to $\Span(\bbG)$ sends:
\begin{enumerate}
\item A backwards map $X \xleftarrow{f} M$ to the restriction $f^\ast\colon \Sp_X \to \Sp_M$;
\item A faithful forwards map $BH\to BG$ to the multiplicative norm $N_H^G\colon \Sp^H \to \Sp^G$;
\item A full forwards map $BG \xrightarrow{g}BK$ to the geometric fixed points $\Phi^{\ker(\pi_1g)}\colon \Sp^G \to \Sp^K$.
\end{enumerate} 
\end{theorem}
\begin{proof}
This is \cite[Theorem 2.18]{integralmodelsyuan}, upon observing that passing to compact objects provides an equivalence from Yuan's category $\mathrm{Glo}^+$ to our category $\Corr$.
\end{proof}

\begin{remark}
Restricted to $\Span(\bbG)$, the functor $\Sp_\bullet$ of \cref{ex:sp_bullet} was constructed in \cite{normsmotivic}, using an $\infty$-categorical construction of the norm $N_H^G\colon \Sp^H \to \Sp^G$. In \cite[Th.8.9]{LLPNorms}, these functors are shown to coherently agree with the Hill--Hopkins--Ravenel point-set model of the norm \cite{hhr}. The additional functoriality of $\Sp_\bullet$ in $\Corr$ encodes further relations present between geometric fixed points and restrictions, such as $\Phi^G\circ q^\ast\simeq\id$ where $q\colon BG \to \ast$ is the projection. Remembering this relation will be crucial to our treatment of the naturality of geometric fixed points in \cref{sec:geometric_fixed_points}.
\end{remark}

Various categories of global spectra and ultracommutative rings may be constructed by taking suitable partially lax limits built from the functor $\Sp_\bullet$. Let us recall this notion.

\begin{mydef}
A \emph{marked category} $(\calI,\calW)$ is a category $\calI$ equipped with a collection of morphisms $\calW\subset\calI$. Given a marked category $(\calI,\calW)$ and a functor $F\colon \calI \to \Cat$, the \emph{partially lax limit}
\[
\underset{(\calI,\calW)}{\laxlim^\dagger} F
\]
is the full subcategory of $\Fun_\calI(\calI,\smallint F)$ of sections of the cocartesian unstraightening of $F$ which send maps in $\calW$ to cocartesian edges.
\end{mydef}

This allows us to define the main characters of this article.

\begin{mydef}
The category of \emph{$\calF$-global spectra} is the partially lax limit
\[\Sp^\gl_\family \coloneqq \underset{(\Glo_\calF^\op,\calO_\calF^\op)}{\laxlim^\dagger}\Sp_\bullet.\]
\end{mydef}

\begin{remark}
By \cite[Th.11.10, Rk.\ on p.\ 1382]{denissilluca}, $\Sp^\gl_\family$ is equivalent to the underlying $\infty$-category of Schwede's category of $\family$-global spectra \cite{s}.
\end{remark}

\begin{mydef}\label{def:choiceofnorms}
An \emph{$\family$-global choice of norms} is a wide subcategory $\calN\subset\Glo_\family$ whose finite coproduct cocompletion $\bbN\subset\bbG_\family$ is closed under base change. The category of \emph{$\calN$-normed $\family$-global rings} is the partially lax limit
\[
\UCom_\family^\calN \coloneqq \underset{(\Span(\bbG_\calF,\bbG_\calF,\bbN),\bbO_\calF^\op)}{\laxlim^\dagger}\Sp_\bullet,
\]
where $\bbO_\calF^\op\subset \bbG_\family^\op\subset\Span(\bbG_\calF,\bbG_\calF,\bbN)$ is the class of backwards arrows that lie in $\bbO_\calF$.
\end{mydef}

Note that as $\Glo_\family\subset\bbG_\family$ is closed under pullbacks along maps in $\calE_\family$, the same is true of the subcategory $\calN\subset\Glo_\family$.

\begin{example}\label{ex:ucom}
When $\calN = \calO_\family$ is the class of faithful morphisms, $\UCom_\family \coloneqq \UCom_\family^{\calO_\family}$ is the category of \emph{ultracommutative $\family$-global rings}. Taking $\calF = \fin$ to be the family of all finite groups, \cite[Th.B]{LLPNorms} proves that $\UCom_\fin$ is equivalent to Schwede's category of $\fin$-global ultracommutative ring spectra \cite[\textsection5]{s}.
\end{example}

\begin{example}\label{ex:nonorms}
When $\calN = \Glo^\simeq_\family$ is the minimal choice of norms, $\UCom_\family^{\calN} \simeq \CAlg(\Sp_\family^\gl)$ is equivalent to the category of $\E_\infty$ rings in $\Sp_\family^\gl$ (see \cref{lem:nonorms}).
\end{example}

\begin{example}
When $\calN = \Glo$ is the maximal choice of norms, $\UCom_\fin^{\Glo}$ is the category of global commutative ring spectra with multiplicative deflations considered by Blumberg--Mandell--Yuan \cite[Df.6.4]{blumbergmandellyuan2023relative}.
\end{example}

\begin{example}
In \cite[Df.2.3]{Barrero_2023}, Barrero introduces the notion of a \emph{global $\family$-transfer system}. Every global $\family$-transfer system $\leq_T$ gives rise to a $\family$-global choice of norms $\calN(T)$ for which $\bbN(T)\subset\bbG_\family$ is the wide subcategory of faithful morphisms $f\colon Y \to X$ satisfying $\pi_1(Y,y) \leq_T \pi_1(X,f(y))$ for all $y \in Y$.  In light of \cref{ex:ucom}, it seems reasonable to expect that the category of $\calN(T)$-normed $\family$-global rings is modeled by a homotopy theory of algebras in orthogonal spectra for a particular global $N_\infty$ operad associated to $\leq_T$ in the sense of \cite{Barrero_global_ninfty}.
\end{example}
We end this section by making good on \cref{ex:nonorms}.

\begin{lemma}\label{lem:nonorms}
Let $\calN = \Glo^\simeq_\family$ be the minimal choice of norms. Then $\UCom_\family^{\calN} \simeq \CAlg(\Sp_\family^\gl)$ is equivalent to the category of $\E_\infty$ rings in $\Sp_\family^\gl$.
\end{lemma}

\begin{proof}
Note that the finite coproduct completion of $\Glo_\family^\simeq$ is given by the wide subcategory $\bbG^\nabla_\family\subset \bbG_\family$ spanned by the fold maps. By \cite[Cor.C.8]{normsmotivic} there is an equivalence 
\[
\underset{\Span(\bbG_\family,\bbG_\family,\bbG_\family^\nabla)}{\laxlim}\Sp_\bullet\simeq\underset{\bbG_\family^\op}{\laxlim}\CAlg(\Sp_\bullet)
\]
between fully lax limits. By inspection, this equivalence lives over the forgetful functor to $\laxlim_{\bbG^{\op}_\family} \Sp_\bullet$, and so restricts to an equivalence
\[
\UCom^{\calN}_{\family} \coloneqq \underset{(\Span(\bbG_\family,\bbG_\family,\bbG_\family^\nabla),\bbO_\family^\op)}{\laxlim^\dagger} \Sp_\bullet \simeq \underset{(\bbG_\family^{\op},\bbO_\family^{\op})}{\laxlim^\dagger} \CAlg(\Sp_\bullet)
\]
between partially lax limits. We then finish with the equivalences
\[
\underset{(\bbG_\family^{\op},\bbO_\family^{\op})}{\laxlim^\dagger} \CAlg(\Sp_\bullet)\simeq \underset{(\Glo_\family^{\op},\calO_\family^{\op})}{\laxlim^\dagger} \CAlg(\Sp_\bullet)\simeq \CAlg(\underset{(\Glo_\family^{\op},\calO_\family^{\op})}{\laxlim^\dagger}\Sp_\bullet) \eqqcolon \CAlg(\Sp^{\gl}_{\calF}),
\]
the first by \cite[Pr.3.6]{linskens2024globalizing} and the second by \cite[Rk.5.1]{denissilluca}.
\end{proof}

%%%%%%%%%%%%%%%%%%%%%%%%%%%%%%%%%%%%%%%%%%%%%%%%%%%%%%%%
%%%%%%%%%%%%%%%%%%%%%%%%%%%%%%%%%%%%%%%%%%%%%%%%%%%%%%%%
%%%%%%%%%%%%%%%%%%%%%%%%%%%%%%%%%%%%%%%%%%%%%%%%%%%%%%%%
\section{Geometric fixed points of equivariant spectra}\label{sec:geometric_fixed_points}
Given a $G$-spectrum $E$ and subgroup $H\subset G$, the geometric fixed points $\Phi^H E$ carry a residual action of the Weyl group $W_HG \cong \Aut_{\Orb_{BG}}(BH\hookrightarrow BG)$. As $H$ varies, this defines a functor
\[
\Phi\colon \Sp^G \to \Fun(\Orb_{BG}^\simeq,\Sp)
\]
categorifying the classical \emph{marks homomorphism} $A(G) \hookrightarrow \Hom(\pi_0 \Orb_{BG}^\simeq,\bbZ)$ for the Burnside ring. Our goal in this section is to refine this to a natural transformation
\[
\Phi\colon \Sp_\bullet \Rightarrow \Fun(\Orb_\bullet^\simeq,\Sp) \colon \Span(\bbG) \to \Cat.
\]
In particular, we construct $\Fun(\Orb_\bullet^\simeq,\Sp)$ as a functor on $\Span(\bbG)$.

\begin{mydef}
Given a symmetric monoidal category $\ccat$, write
\[
\underline{\ccat}^\otimes\colon \Span(\bbG)\to\Cat
\]
for the composite
\begin{center}\begin{tikzcd}
\Span(\bbG)\ar[r,"\Span(\pi_0)"]&\Span(\bbF)\ar[r,"\ccat^\otimes"]&\Cat,
\end{tikzcd}\end{center}
where $\ccat^\otimes$ is the product-preserving functor encoding the symmetric monoidal product on $\ccat$.
\end{mydef}

\begin{prop}\label{prop:right_kan_of_const}
Let $\bbR\subset\bbG_\family$ be any class of morphisms which is closed under base change and $\ccat$ be a symmetric monoidal category. Then the value at $X$ of the right Kan extension of $\ul{\ccat}^\otimes$ along the inclusion
\[
\iota\colon \Span(\bbG_\family,\bbE_\family,\bbR)\to\Span(\bbG_\family,\bbG_\family,\bbR)
\]
is given by $\Fun(\Orb_X^\simeq,\ccat)$. In particular, this right Kan extension again preserves finite products.
\end{prop}
\begin{proof}
Consider the following diagram:
\begin{center}\begin{tikzcd}
(\Glo_\family^\op)^\simeq\ar[r,hook]\ar[d,"i"]&\mathcal{E}_\family^\op\ar[r]\ar[d,"j"]&\bbE_\family^\op\ar[r,hook]\ar[d,"j'"]&\Span(\bbG_\family,\bbE_\family,\bbR)\ar[d]\\
\calO_\family^\op\ar[r,hook]&\Glo_\family^\op\ar[r]&\bbG_\family^\op\ar[r,hook]& \Span(\bbG_\family,\bbG_\family,\bbR)
\end{tikzcd}\end{center}
We first claim that the Beck--Chevalley transformations relating horizontal restrictions of vertical right Kan extensions for each of these squares is an equivalence. For the rightmost square, this follows from \cite[Lm.4.1.6]{CHLLBispans} applied to the factorization systems $(\Span(\bbG_\family,\bbE_\family,\bbR),\bbE_\family^\op,\bbR)\subset(\Span(\bbG_\family,\bbG_\family,\bbR),\bbG_\family^\op,\bbR)$. For the middle square, it holds as $\ul{\ccat}^\otimes$ preserves finite products and the middle square realizes $j'$ as the finite product completion of $j$. Moreover, this implies that the right Kan extension of any finite product-preserving functor on $\bbE_\family^\op$ along $j$ again preserves finite products. For the leftmost square, it follows from (the dual of) \cite[Pr.4.1.4]{CHLLBispans}.

As $\iota_\ast\ul{\calC}^\otimes$ preserves finite products, to compute its value on a finite groupoid $X$ we may reduce to the case where $X$ is connected. By the above discussion, it therefore suffices to restrict $\ul{\ccat}^\otimes$ to $(\Glo_\family^\op)^\simeq$ and compute its right Kan extension along $i$. As $\ul{\ccat}^\otimes$ restricted to $(\Glo_\family^\op)^\simeq$ is the constant functor on $\ccat$, the pointwise formula for right Kan extensions allows us to identify this right Kan extension as the constant limit
\[
i_\ast \underline{\ccat}^\otimes(X)\simeq \lim_{(\Glo_\family^\op)^\simeq\underset{\calO_\family^\op}{\times}((\calO_{\family})_{/X})^\op}\ccat \simeq \lim_{\Orb_X^\simeq}\ccat\simeq\Fun(\Orb_X^\simeq,\ccat)
\]
for any $X \in \Glo$, establishing the proposition.
\end{proof}

Taking $\family$ to be the family of all finite groups and $\bbR = \bbG$, this produces for any symmetric monoidal category $\ccat$ a functor
\[
\Fun(\Orb_\bullet^\simeq,\ccat) = \iota_\ast\ul{\ccat}^\otimes \colon \Span(\bbG)\to\Cat.
\]

\begin{remark}
Let $f\colon Y \to X$ be a map of finite groupoids. Tracing through the construction, one finds that the restriction and norm maps
\[
f^\ast\colon \Fun(\Orb_X^\simeq,\ccat)\to\Fun(\Orb_Y^\simeq,\ccat),\qquad f_\otimes\colon \Fun(\Orb_Y^\simeq,\ccat)\to\Fun(\Orb_X^\simeq,\ccat)
\]
encoded by the functoriality of $\Fun(\Orb_\bullet^\simeq,\ccat)$ in $\Span(\bbG)$ may be described as follows.
For the restriction, given an object $M \in \Orb_Y$, the composite $M \to Y \to X$ factors uniquely as a full functor $M \to M'$ followed by a faithful functor $M' \to X$. This determines a functor $f_!\colon \Orb_Y \to \Orb_X$, and we have
\[
f^\ast\colon \Fun(\Orb_X^\simeq,\ccat)\to\Fun(\Orb_Y^\simeq,\ccat),\qquad (f^\ast E)(M) = E(f_! M).
\]
For the norm, we have
\[
f_\otimes\colon \Fun(\Orb_Y^\simeq,\ccat)\to\Fun(\Orb_X^\simeq,\ccat),\qquad (f_\otimes E)(M) \simeq \bigotimes_{S \subset Y\times_X M}E(S),
\]
this tensor product being over the path components of $ Y\times_X M$.
\end{remark}

We now focus on the case $\ccat = \Sp$, where this construction produces a natural recipient for the geometric fixed points of equivariant spectra.

\begin{lemma}\label{prop:funct_geom_fixed_points}
Restricted to $\Span(\bbG,\bbE,\bbG)$, there is a natural transformation
\[
\Phi^\bullet\colon \Sp_\bullet\to\underline{\Sp}^\otimes\colon \Span(\bbG,\bbE,\bbG)\to\Cat
\]
whose value on $BG\in\Glo$ is given by the geometric fixed points functor $\Phi^G\colon \Sp^G\to\Sp$.
\end{lemma}
\begin{proof}
For any finite groupoid $X$, right Kan extension along $X \to \pi_0 X$ defines a functor
\[
(\bs)^X\colon \bbF_X \to \bbF_{\pi_0 X}
\]
taking fixed points on each component of $X$. As fixed points commute with restrictions along surjective maps and right Kan extensions along arbitrary morphisms, this defines a natural transformation between the two functors
\[
\Span(\bbG,\bbE,\bbG) \to \Corr,\qquad X \mapsto \bbF_X\text{ and }X \mapsto \bbF_{\pi_0 X}.
\]
To avoid a clash of notation, let us temporarily write $\widetilde{\Sp}_\bullet\colon \Corr\to \Cat$ for the functor $\Sp_\bullet$ of \cref{ex:sp_bullet}. Then 
\[
\Sp_\bullet = \widetilde{\Sp}_\bullet\circ \bbF_{(\bs)},\qquad \underline{\Sp}^\otimes = \widetilde{\Sp}_\bullet \circ \bbF_{\pi_0(\bs)}
\]
as functors on $\Span(\bbG)$, and so the natural transformation $(\bs)^X\colon \bbF_X\to\bbF_{\pi_0X}$ determines the desired natural transformation $\Phi^\bullet\colon \Sp_\bullet\to\underline{\Sp}^\otimes$.
\end{proof}

\begin{theorem}\label{prop:weylaction}
There is a natural symmetric monoidal transformation
\[
\Phi\colon \Sp_\bullet\Rightarrow\iota_\ast\ul{\Sp}^\otimes\simeq\Fun(\calO_\bullet^\simeq,\Sp)\colon \Span(\bbG)\to\Cat
\]
whose value on a $G$-spectrum $E$ is the functor
\[
\Phi E\colon \Orb_{BG}^\simeq\to\Sp,\qquad (BH\hookrightarrow BG) \mapsto \Phi^H E.
\]
\end{theorem}
\begin{proof}
This natural transformation is adjoint to the transformation of \cref{prop:funct_geom_fixed_points}. The identification of its values follows from \cref{prop:right_kan_of_const}. Symmetric monoidality is automatic by \cite[Pr.C.9]{normsmotivic}.
\end{proof}

%%%%%%%%%%%%%%%%%%%%%%%%%%%%%%%%%%%%%%%%%%%%%%%%%%%%%%%%
%%%%%%%%%%%%%%%%%%%%%%%%%%%%%%%%%%%%%%%%%%%%%%%%%%%%%%%%
%%%%%%%%%%%%%%%%%%%%%%%%%%%%%%%%%%%%%%%%%%%%%%%%%%%%%%%%
\section{Geometric fixed points of ultracommutative rings}\label{sec:funwithlaxlims}
Applying partially lax limits to \cref{prop:weylaction} provides a functor
\begin{equation}\label{eq:laxlimphi}
\UCom_\family^\calN\simeq\underset{(\Span(\bbG_\family,\bbG_\family,\bbN),\bbO_\family^\op)}{\laxlim^\dagger}\Sp_\bullet \xrightarrow{\laxlim^\dagger\Phi}\underset{(\Span(\bbG_\calF,\bbG_\calF,\bbN),\bbO_\calF^\op)}{\laxlim^\dagger} \iota_*\underline{\Sp}^{\otimes}
\end{equation}
for any $\family$-global choice of norms $\calN$. Our next goal is to identify the partially lax limit on the right hand side.

We start by describing the partially lax limit of a right Kan extension in a special case. Given cocartesian fibrations $\calF,\calF'\to\ccat$, write $\Fun^{\mathrm{co}}_\ccat(\calF,\calF')$ for the category of functors $\calF\to\calF'$ over $\ccat$ which preserve cocartesian edges.

\begin{lemma}\label{lem:envelopes}
Let $(\calC,L,R)$ be a category equipped with a factorization system. Then $t\colon \Ar_R(\ccat)\to\ccat$ is a cocartesian fibration, and there is an equivalence
\[
\underset{(\ccat,L)}{\laxlim^\dagger}F\simeq\Fun^{{\mathrm{co}}}_{\ccat}(\Ar_R(\ccat),\smallint F)
\]
for any functor $F\colon \ccat\to\Cat$.
\end{lemma}
Here, $\smallint F$ is the cocartesian unstraightening of $F$ and $\Ar_R(\ccat)\subset \Fun(\Delta^1,\ccat)$ is the full subcategory of arrows in $R$, with $t(x\to y) = y$.
\begin{proof}
This is an immediate consequence of \cite[Proposition 2.2.4]{envelopes}.
\end{proof}

\begin{prop}\label{prop:laxlimrkan}
Let $(\ccat,L,R)$ be a category equipped with a factorization system. Let $i\colon \dcat\subset\ccat$ be a subcategory containing $R$, and suppose that $(\dcat,\dcat\cap L,R)$ is again a factorization system. Then there is an equivalence
\[
\underset{(\ccat,L)}{\laxlim^\dagger}i_\ast F\simeq \underset{(\dcat,\dcat\cap L)}{\laxlim^\dagger}F
\]
for any functor $F\colon \dcat\to\Cat$, where $i_\ast F$ is the right Kan extension of $F$ along $i$.
\end{prop}
\begin{proof}
The assumption that $(\dcat,\dcat\cap L,R)$ is a factorization system guarantees that
\begin{center}\begin{tikzcd}
\Ar_R(\dcat)\ar[r,hook]\ar[d,"t"]&\Ar_R(\ccat)\ar[d,"t"]\\
\dcat\ar[r,"i",hook]&\ccat
\end{tikzcd}\end{center}
is cartesian. Using \cref{lem:envelopes}, we may therefore compute
\begin{align*}\hspace{-9.66pt}
\underset{(\ccat,L)}{\laxlim^\dagger}i_\ast F \simeq \Fun^{\text{co}}_\ccat(\Ar_R(\ccat),\smallint i_\ast F)\simeq \Fun^{{\text{co}}}_\dcat(i^\ast\Ar_R(\ccat),\smallint F)\simeq \Fun^{{\text{co}}}_\dcat(\Ar_R(\dcat),\smallint F) \simeq \underset{(\dcat,\dcat\cap L)}{\laxlim^\dagger}F
\end{align*}
as claimed.
\end{proof}

We will ultimately use this to reduce our partially lax limit to the following example.

\begin{lemma}\label{lm:calg_as_partiallylaxlimit}
Let $\calC^\otimes\colon \Span(\Fin)\to \Cat$ be a symmetric monoidal category. Let $I$ be a category, and denote by $I^\amalg \to \Span(\Fin)$ the cocartesian operad on $I$, and by $I^\amalg_\inert$ the collection of inert edges. Then there is an equivalence
\[
\Fun(I,\CAlg(\calC^\otimes)) \simeq \underset{(I^\amalg,I^\amalg_{\inert})}{\laxlim^{\dagger}}\left(I^\amalg \to \Span(\Fin) \xrightarrow{\calC^\otimes} \Cat\right).
\]
\end{lemma}

\begin{proof}
By definition, the right hand side is equivalent to the subcategory of functors
\[\begin{tikzcd}
	{I^\amalg} && {\smallint \calC^\otimes} \\
	& {\Span(\Fin)}
	\arrow["F", from=1-1, to=1-3]
	\arrow[from=1-1, to=2-2]
	\arrow[from=1-3, to=2-2]
\end{tikzcd}\]
over $\Span(\Fin)$ which send inert edges in $I^\amalg$ to cocartesian edges of the cocartesian unstraightening $\smallint \calC^\otimes$. This is equivalent to $\Fun(I,\CAlg(\calC^\otimes))$ by the universal property of the cocartesian operad \cite[Th.2.4.3.18]{ha}.
\end{proof}

A good supply of cocartesian operads is obtained from the following.

\begin{lemma}\label{prop:spancocartoperad}
Let $(\ccat,L,R)$ be a category equipped with wide subcategories $L,R\subset\ccat$ closed under pullback along each other, and let $(\bbC,\bbL,\bbR)$ be its finite coproduct completion. Then the functor
\[
\Span(\pi_0)\colon \Span(\bbC,\bbL,\bbR) \to \Span(\bbF)
\]
exhibits the source as the cocartesian operad on $\Span(\ccat,L,R)$. Moreover, the inert morphisms are given by the class $(\bbC^\nabla)^\op\subset\bbL^\op\subset\Span(\bbC,\bbL,\bbR)$ of backwards fold maps.
\end{lemma}

\begin{proof}
We will make use of a characterization of cocartesian operads to appear in forthcoming work of Cnossen, see \cite[Cor.16.2.17]{ha_notes}. Combined with the identification of operads with $(\Span(\bbF),\bbF^{\op},\bbF)$-operads in the sense of \cite[Def.2.46]{LLPNorms}, see Ex.2.47 there, this shows that a functor $p\colon \dcat\to\Span(\bbF)$ exhibits $\dcat$ as the cocartesian operad of its fiber $\dcat_1$ if and only if the following conditions are satisfied:
\begin{enumerate}
\item $\dcat$ admits $p$-cocartesian lifts of maps in $\bbF^{\op}\subset \Span(\bbF)$;
\item $\dcat$ is semiadditive;
\item $p$ preserves finite products;
\item For all $X,Y\in\dcat$, the projections $X\times Y\to X,Y$ are $p$-cocartesian.
\end{enumerate}
Moreover, in this case the inert morphisms in $\dcat$ are the $p$-cocartesian morphisms of (1).

We claim that $\Span(\pi_0)\colon \Span(\bbC,\bbL,\bbR)\to\Span(\bbF)$ satisfies these conditions. As $\pi_0\colon \bbL \to \bbF$ is a cartesian fibration, \cite[Th.3.1]{HHLNa} shows that $\pi_0$-cartesian lifts along $\bbL \to \bbF$, viewed as backwards maps, define $\Span(\pi_0)$-cocartesian lifts of maps in $\bbF^\op\subset\Span(\bbF)$, proving (1). As the $\pi_0$-cartesian arrows are precisely the fold maps, this also shows that the $\Span(\pi_0)$-cocartesian lifts of backwards maps are precisely the backwards fold maps $(\bbC^\nabla)^{\op}\subset \Span(\bbC,\bbL,\bbR)$, giving the claimed identification of inert morphisms.

As the categories $\bbL,\bbR\subset\bbC$ admit finite coproducts, the category $\Span(\bbC,\bbL,\bbR)$ is semiadditive with direct sums computed as coproducts in the underlying category $\bbC$; see \cite[Pr.2.2.5]{CHLLTambara} for a proof in the case $\bbL = \bbC$ which works just as well here. This establishes (2) and (3). This result also shows that the projection maps are given by backwards coproduct inclusions, which together with the above identification of inerts establishes (4).
\end{proof}

This has all led to the following computation.

\begin{prop}\label{prop:laxlimasfunctorcategory}
There is an equivalence
\[
\underset{(\Span(\bbG_\family,\bbG_\family,\bbN),\bbO_\family^\op)} {\laxlim^\dagger}\iota_\ast \underline{\ccat}^\otimes\simeq \Fun(\Span(\Glo_\family,\mathcal E_\family,\calN),\CAlg(\ccat)).
\]
for any $\family$-global choice of norms $\calN$ and symmetric monoidal category $\ccat$.
\end{prop}
\begin{proof}
Let $\bar{\bbE}_\family\subset\bbG_\family$ denote the class of morphisms with connected fibers, so that $(\bar{\bbE}_\family,\bbO_\family)$ is a factorization system on $\bbG_\family$. Then $\bbO_\family^\op$ is the left class of a factorization system on $\Span(\bbG_\family,\bbG_\family,\bbN)$ with right class $\bbN\circ \bar{\bbE}_\family^\op$ given by those spans with backwards map in $\bar{\bbE}_\family$. This right class is contained in the subcategory $\Span(\bbG_\family,\bbE_\family,\bbN)$, where it is the right class in a factorization system with left class
\[
\bbO_\family^\op\cap\Span(\bbG_\family,\bbE_\family,\bbN)\simeq(\bbG_\family^\nabla)^\op
\]
equivalent to the class of backwards fold maps. We therefore obtain equivalences
\[
\underset{(\Span(\bbG_\family,\bbG_\family,\bbN),\bbO_\family^\op)} {\laxlim^\dagger}\iota_\ast \underline{\ccat}^\otimes \simeq  \underset{(\Span(\bbG_\family,\bbE_\family,\bbN),(\bbG_\family^\nabla)^\op)}{\laxlim^\dagger}\underline{\ccat}^\otimes \simeq  \Fun(\Span(\Glo_\family,\mathcal E_\family,\calN),\CAlg(\ccat)),
\]
the first by  \cref{prop:laxlimrkan} and the second by combining \cref{lm:calg_as_partiallylaxlimit} and \cref{prop:spancocartoperad}.
\end{proof}

The primary case of interest is where $\ccat = \Sp$. Combining the equivalence of \Cref{prop:laxlimasfunctorcategory} with (\ref{eq:laxlimphi}), we obtain a geometric fixed points functor
\[
\Phi\colon \UCom_\family^\calN \to \underset{(\Span(\bbG_\family,\bbG_\family,\bbN),\bbO_\family^\op)} {\laxlim^\dagger}\iota_\ast \underline{\Sp}^\otimes\simeq \Fun(\Span(\Glo_\family,\mathcal E_\family,\calN),\CAlg),
\]
This establishes the following:

\begin{theorem}\label{thm:geometricfpofucomm}
Geometric fixed points of $\calN$-normed $\family$-global rings assemble into the functor
\[
\Phi\colon \UCom_\family^\calN\to\Fun(\Span(\Glo_\family,\calE_\family,\calN),\CAlg).
\]
Its value on a $\calN$-normed $\family$-global ring $E$ is the functor
\[
\Phi E\colon \Span(\Glo_\family,\mathcal E_\family,\calN) \to \CAlg,\qquad BG \mapsto \Phi^G E,
\]
with functoriality in $\mathcal E_\family^\op$ given by restriction and in $\calN$ given by geometric norms.
\qed
\end{theorem}

%%%%%%%%%%%%%%%%%%%%%%%%%%%%%%%%%%%%%%%%%%%%%%%%%%%%%%%%
%%%%%%%%%%%%%%%%%%%%%%%%%%%%%%%%%%%%%%%%%%%%%%%%%%%%%%%%
%%%%%%%%%%%%%%%%%%%%%%%%%%%%%%%%%%%%%%%%%%%%%%%%%%%%%%%%
\section{Rational ultracommutative rings}\label{sec:rational}
We can now put everything together to prove our main theorem. First, we review the requisite notion of rationality. 

\begin{mydef}
Let $R\subset\Q$ be a subring of the rational numbers. We say that an $X$-spectrum (resp.\ $\family$-global spectrum) $E$ is \emph{$R$-local} if any of the following equivalent conditions are satisfied:
\begin{enumerate}
\item $\pi_\ast E^H$ is an $R$-module for every $BH\to X$ in $\calO_X$ (resp.\ every $H \in \family$);
\item $\pi_\ast \Phi^H E$ is an $R$-module for every $BH \to X$ in $\calO_X$ (resp.\ every $H\in \family$);
\item $E$ is Bousfield local with respect to the localization $\spherespectrum[(R^\times\cap\Z)^{-1}]$.
\end{enumerate}
The full subcategory of $R$-local $X$-spectra is denoted by $\Sp_{X,R}\subset\Sp_{X}$, and the full subcategory of $R$-local $\family$-global spectra is denoted by $\Sp^\gl_{\family,R}\subset\Sp^\gl_\family$.
\end{mydef}

As $R$-locality is detected by geometric fixed points, restriction to $R$-local equivariant spectra determines a subfunctor
\[
\Sp_{\bullet,R}\subset\Sp_\bullet\colon \Span(\bbG) \to \Cat
\]
for which the natural transformation $\Phi$ of \cref{prop:weylaction} restricts to
\[
\Phi\colon\Sp_{\bullet,R}\Rightarrow\iota_\ast\ul{\Sp}_R^\otimes.
\]

\begin{prop}\label{prop:wimmer}
Let $\family$ be a family of finite groups and $R\subset \Q$ be a subring with $|G| \in R^\times$ for all $G\in \family$. Then $\Phi$ restricts to an equivalence
\[
\Phi\colon \Sp_{\bullet,R}\simeq \iota_\ast \ul{\Sp}_R^\otimes\colon \Span(\bbG_\family)\to\Cat.
\]
\end{prop}
\begin{proof}
As described in \cref{prop:weylaction}, this natural transformation has as its components products of the functors
\[
\Phi\colon \Sp_{BG,R}\to\Fun(\Orb_{BG}^\simeq,\Sp_R),\qquad (\Phi E)(BH\hookrightarrow BG)  = \Phi^H E.
\]
It is folklore that this is an equivalence when $|G| \in R^\times$; see \cite[Th.3.10]{wimmerrationalmodel} for a proof.
\end{proof}

Write $\UCom_{\family,R}^\calN\subset\UCom_\family^\calN$ for the full subcategory of $\calN$-normed $\family$-global rings whose underlying $\family$-global spectrum is $R$-local. We may now prove our main theorem.

\begin{theorem}\label{thm:rationalucomm}
Let $\family$ be a family of finite groups and $R\subset\Q$ be a subring satisfying $|G| \in R^\times$ for all $G \in \family$. Then the functor $\Phi$ of \Cref{thm:geometricfpofucomm} restricts to an equivalence
\[
\UCom_{\family,R}^\calN\simeq \Fun(\Span(\Glo_\family,\mathcal E_\family,\calN),\CAlg_R).
\]
\end{theorem}
\begin{proof}
It is immediate from the definitions that
\[
\UCom_{\family,R}^\calN\simeq\underset{\Span(\bbG_\calF,\bbG_\calF,\bbN),\bbO_\calF^\op)}{\laxlim^\dagger}\Sp_{\bullet,R}.
\]
The restriction of $\Phi$ to $R$-local objects may therefore be identified as the composite
\[
\underset{(\Span(\bbG_\calF,\bbG_\calF,\bbN),\bbO_\calF^\op)}{\laxlim^\dagger}\Sp_{\bullet,R}\to\underset{\Span(\bbG_\calF,\bbG_\calF,\bbN),\bbO_\calF^\op)}{\laxlim^\dagger}\iota_\ast\ul{\Sp}_R^\otimes \simeq\Fun(\Span(\Glo_\family,\mathcal E_\family,\calN),\CAlg_R),
\]
where the first map is obtained from the natural transformation $\Phi\colon \Sp_{\bullet,R}\to\iota_\ast\underline{\Sp}_R^\otimes$ and is an equivalence by \cref{prop:wimmer}, and the second equivalence is \cref{prop:laxlimasfunctorcategory}. Combining these equivalences proves the theorem.
\end{proof}

%%%%%%%%%%%%%%%%%%%%%%%%%%%%%%%%%%%%%%%%%%%%%%%%%%%%%%%%
%%%%%%%%%%%%%%%%%%%%%%%%%%%%%%%%%%%%%%%%%%%%%%%%%%%%%%%%
%%%%%%%%%%%%%%%%%%%%%%%%%%%%%%%%%%%%%%%%%%%%%%%%%%%%%%%%
\section{Variations on the theme}\label{sec:variations}
The strategy used to prove our main theorem may also be applied to identify rational objects in other equivariant categories. We illustrate this with two examples: a linear version for plain global spectra, and a fixed-$G$ version for $G$-commutative ring spectra.

Our first example is a variant of Barrero--Barthel--Pol--Strickland--Williamson's identification of rational $\family$-global spectra \cite[Th.D]{Glo_reps}; also see earlier work of Schwede \cite[Th.4.5.13]{s} and Wimmer \cite[Th.3.2.20]{wimmerthesis}.

\begin{theorem}\label{thm:rational_global_linear}
Geometric fixed points for $\family$-global spectra assemble into a symmetric monoidal functor
\[
\Phi\colon \Sp^\gl_\family\to\Fun(\mathcal E_\family^\op,\Sp).
\]
This functor restricts to an equivalence between full subcategories of $R$-local objects for any subring $R\subset\Q$ satisfying $|G| \in R^\times$ for all $G\in \family$.
\end{theorem}
\begin{proof}
This functor is obtained as a composite
\[
\Sp^\gl_\family = \underset{(\Glo_\family^\op,\calO_\family^\op)}{\laxlim^\dagger}\Sp_\bullet \xrightarrow{\laxlim^\dagger\Phi}\underset{(\Glo_\family^\op,\calO_\family^\op)}{\laxlim^\dagger}\iota_\ast\ul{\Sp}^\otimes \simeq \Fun(\mathcal E_\family^\op,\Sp).
\]
Here, the first functor is obtained from the natural transformation of \cref{prop:weylaction}, and is therefore an equivalence when restricted to $R$-local objects by \cref{prop:wimmer}. The latter equivalence is a consequence of \cref{prop:laxlimrkan} applied to the factorization system $(\Glo_\family^\op,\mathcal O_\family^\op,\calE_\family^\op)$ and the subcategory $\calE_\family^\op \subset\Glo_\family^\op$, which implies
\[
\underset{(\Glo_\family^\op,\calO_\family^\op)}{\laxlim^\dagger}\iota_\ast\ul{\Sp}^\otimes \simeq \underset{(\calE_\family^\op,(\Glo_\family^\op)^\simeq)}{\laxlim^\dagger}\ul{\Sp}^\otimes\simeq \underset{\calE_\family^\op}{\laxlim}\Sp\simeq \Fun(\mathcal E_\family^\op,\Sp),
\]
as $\ul{\Sp}^\otimes$ restricts to the constant diagram over $\calE_\family^\op$.
\end{proof}

Our second example requires a final set of definitions. As $\bbF_X$ is equivalent to $\bbO_{/X}$, the projection defines a canonical pullback-preserving functor $\bbF_X\simeq\bbO_{/X}\to\bbG$. In particular, we may restrict $\Sp_\bullet$ to a functor defined on $\Span(\bbF_X)$.

\begin{mydef}[{\cite[Df.1.2, Pr.3.6]{blumberghill2018incomplete}}]
Let $X$ be a finite groupoid. An \emph{indexing system} on $X$ is a wide subcategory $\calI\subset\calO_X$ whose finite coproduct completion $\bbI\subset\bbF_X$ is stable under base change.
\end{mydef}

\begin{mydef}
The category of \emph{$\calI$-normed $X$-ring spectra} is the partially lax limit
\[
\Comm_X^\calI \coloneqq \underset{(\Span(\bbF_X,\bbF_X,\bbI),\bbF_X^\op)}{\laxlim^\dagger}\Sp_\bullet.
\]
\end{mydef}

\begin{ex}
When $X = BG$ and $\calI = \calO_{BG}$ is the maximal indexing system, $\Comm^G \coloneqq \Comm_{BG}^{\calO_{BG}}$ is the category of \emph{$G$-commutative rings} or \emph{normed $G$-ring spectra}. By \cite[Th.A]{LLPNorms}, $\Comm^G$ is equivalent to the underlying $\infty$-category of Hill--Hopkins--Ravenel's homotopy theory of strictly $G$-commutative ring spectra \cite{hhr}.
\end{ex}

\begin{ex}
When $\calI = \calO_X^\simeq$ is the minimal indexing system, $\Comm_X^{\calI}\simeq\CAlg(\Sp_X)$ is equivalent to the category of $\E_\infty$ rings in $\Sp_X$; the proof is identical to that of \cref{lem:nonorms}.
\end{ex}

Our methods can be used to establish a generalization of Wimmer's identification of rational $G$-commutative rings \cite[Cor.1.4]{wimmerrationalmodel}. This generalization has also been established independently by Tigilauri \cite{tigilauri2026modelnormedalgebrasrational}. First, we need a lemma. Write $\bbF_X^\nabla\subset\bbF_X$ for the class of fold maps.

\begin{lemma}\label{lem:bc}
Let $\calI$ be an indexing system on a finite groupoid $X$, and consider the diagram
\begin{center}\begin{tikzcd}[column sep=tiny, row sep=small]
&(\Orb_X^\simeq)^\op\ar[rr]\ar[dd]\ar[dl]\ar[dr]&&(\bbF_X^\nabla)^\op\ar[dd]\ar[dr]\ar[rr]&&\Span(\bbF_X,\bbF_X^\nabla,\bbI)\ar[dd]\ar[dr]\\
(\Glo^\simeq)^\op\ar[rr]\ar[dd]&&\calE^\op\ar[dd]\ar[rr]&&\bbE^\op\ar[rr]\ar[dd]&&\Span(\bbG,\bbE,\bbG)\ar[dd]\\
&\Orb_X^\op\ar[rr]\ar[dl]\ar[dr]&&\bbF_X^\op\ar[dr]\ar[rr]&&\Span(\bbF_X,\bbF_X,\bbI)\ar[dr]\\
\calO^\op\ar[rr]&&\Glo^\op\ar[rr]&&\bbG^\op\ar[rr]&&\Span(\bbG)
\end{tikzcd},\end{center}
For any product-preserving functor $F$ on $\Span(\bbG,\bbE,\bbG)$, the Beck--Chevalley transformations relating the vertical right Kan extensions of its restrictions are all equivalences.
\end{lemma}
\begin{proof}
For the front faces, this was seen in the proof of \cref{prop:right_kan_of_const}, and the same argument applies to the back faces. This reduces us to showing that the Beck--Chevalley transformation associated to the leftmost diagonal face is an equivalence, which follows from the pointwise formula for Kan extensions as the comparison $(\Orb_X^\simeq)\times_{\Orb_X}((\Orb_X)/Y) \to \Glo^\simeq\times_{\calO}\calO_{/Y}$ is an equivalence for any $Y \in \Orb_X$.
\end{proof}

\begin{theorem}\label{thm:ultracommutative_rational_Xrings}
Let $X$ be a finite groupoid and $\calI$ be an indexing system on $X$. Then the geometric fixed points of $\calI$-normed $X$-ring spectra assemble into a functor
\[
\Phi\colon \Comm_X^\calI \to \Fun(\calI,\CAlg).
\]
This functor restricts to an equivalence between full subcategories of $R$-local objects for any subring $R\subset \Q$ satisfying $|\pi_1(X,x)| \in R^\times$ for all $x\in X$.
\end{theorem}
\begin{proof}
This functor is obtained as a composite
\[
\Comm_X^\calI = \underset{(\Span(\bbF_X,\bbF_X,\bbI),\bbF_X^\op)}{\laxlim^\dagger}\Sp_\bullet
\xrightarrow{\laxlim^\dagger\Phi}
\underset{(\Span(\bbF_X,\bbF_X,\bbI),\bbF_X^\op)}{\laxlim^\dagger}\iota_\ast\ul{\Sp}^\otimes \simeq \Fun(\calI,\CAlg).
\]
Here, the first functor is obtained from the natural transformation of \cref{prop:weylaction}, and is therefore an equivalence when restricted to $R$-local objects by \cref{prop:wimmer}.

The final equivalence is obtained as follows. By \cref{lem:bc}, we may identify $\iota_\ast\ul{\Sp}^\otimes$ restricted to $\Span(\bbF_X,\bbF_X,\bbI)$ as the corresponding right Kan extension of $\ul{\Sp}^\otimes$ restricted to $\Span(\bbF_X,\bbF_X^\nabla,\bbI)$. Consider the identifications
\[
\underset{(\Span(\bbF_X,\bbF_X,\bbI),\bbF_X^\op)}{\laxlim^\dagger}\iota_\ast\ul{\Sp}^\otimes\simeq\underset{(\Span(\bbF_X,\bbF_X^\nabla,\bbI),(\bbF_X^\nabla)^\op)}{\laxlim^\dagger}\ul{\Sp}^\otimes \simeq \Fun(\Span(\Orb_X,\Orb_X^\simeq,\calI),\CAlg),
\]
the first following from \cref{prop:laxlimrkan} and the second by a combination of \cref{lm:calg_as_partiallylaxlimit} together with \cref{prop:spancocartoperad}. We conclude as $\Span(\Orb_X,\Orb_X^\simeq,\calI) \simeq \calI$.
\end{proof}

\addcontentsline{toc}{section}{References}
\scriptsize
\bibliography{reference}
\bibliographystyle{alpha}
\end{document}